\documentclass[11pt,reqno]{amsart}
\usepackage{geometry}                
\usepackage{graphicx}
\usepackage{amssymb,amsmath,amsfonts,latexsym}
\usepackage[all, cmtip]{xy}
\usepackage{mathtools}
\usepackage{tikz-cd}

\newtheorem{theorem}{Theorem}[section]

\newtheorem{proposition}[theorem]{Proposition}
\newtheorem{definition}{Definition}[section]
\newtheorem{remark}{Remark}[section]
\newtheorem{corollary}{Corollary}[section]
\newtheorem{example}{Example}[section]

\def \C  {\mathbb C}
\def \R  {\mathbb R}

\def \I  {\widetilde{\mathfrak m}}

\begin{document}

\title{Sections of Hamiltonian Systems}
\author{Konstantinos Kourliouros}
\address{K. Kourliouros: ICMC-USP, Av. Trabalhador Sancarlense, 400-Centro, S\~ao Carlos, S\~ao Paulo, Brasil.}
\email{k.kourliouros@gmail.com}

\maketitle

\begin{abstract}
A section of a Hamiltonian system is a hypersurface in the phase space of the system, usually representing a set of one-sided constraints (e.g. a boundary, an obstacle or a set of admissible states). In this paper we give local classification results for all typical singularities of sections of regular (non-singular) Hamiltonian systems, a problem equivalent to the classification of typical singularities of Hamiltonian systems with one-sided constraints. In particular we give a complete list of exact normal forms with functional invariants, and we show how these are related/obtained by the symplectic classification of mappings with prescribed (Whitney-type) singularities, naturally defined on the reduced phase space of the Hamiltonian system.    
\medskip

\noindent \textbf{Keywords:} Hamiltonian Systems, Constraints, Singularities, Normal Forms, Functional Moduli.

\noindent \textbf{MSC2020:} 34C20, 37J06, 57R45, 70H15, 70H45.
\end{abstract}

\section{Introduction}
\label{sec1}

By a Hamiltonian system we mean a triple $(X,\omega,f)$ where $(X,\omega)$ is a symplectic manifold and $f:X\rightarrow \mathbb{R}$ is a function (the Hamiltonian). By a section of the Hamiltonian system $(X,\omega,f)$ we mean simply a smooth hypersurface $H\subset X$ (a submanifold of codimension 1). In most cases this can be interpreted, depending on the application, as a boundary or as an obstacle, lifted from the configuration space of the system $M$ to its phase space $X=T^*M$. This is the case for example in  boundary value problems of differential operators, as the problem of diffraction in geometric optics, the billiard ball problem, the problem of bypassing an obstacle e.t.c. (c.f. \cite{A1}, \cite{A11}, \cite{A2}, \cite{A3}, \cite{A4}, \cite{M}, \cite{M1}). In different applications (c.f. \cite{A0}, \cite{D}, \cite{L}, \cite{Pn}) it can be interpreted also as a set of admissible states, in the sense that the motion of the system is confined to lie strictly on the submanifold $H$. In any case, giving a section of a Hamiltonian system is equivalent to defining a Hamiltonian system with one-sided constraints, i.e. a quadruple $(X,\omega,f,H)$, where the constraints are represented by the hypersurface $H$. 

In this paper we deal with the problem of local classification of typical singularities of sections $H$ of regular (i.e. non-singular) Hamiltonian systems $(X, \omega,f)$ by diffeomorphisms preserving the system itself, i.e. by symplectomorphisms of $\omega$ which also preserve the Hamiltonian function $f$ (equivalently, the Hamiltonian vector field $Z_f$).  Definitely, the problem is equivalent to the local classification of Hamiltonian systems with one-sided constraints $(X,\omega,f,H)$ under the whole ``group" of local diffeomorphisms.

In the author's knowledge, this problem was first posed by R. B. Melrose in \cite{M}, as being of considerable importance in the theory of boundary value problems of (pseudo-) differential operators, and thus in the theory of wave propagation, in the construction of parametrices along gliding rays e.t.c. One may also consider it as a problem of classification of families of systems of solutions (Lagrangian manifolds-singularities c.f. \cite{A1}-\cite{A4}) in general variational problems with one-sided constraints, the parameter of the family being the value $f=t$ of the energy. 

In \cite{M} Melrose studied the closely related problem of local classification of pairs of hypersurfaces in symplectic space $(X,\omega)$, i.e. of pairs $(H,X_0)$, where $X_0=\{f=0\}$ is a fixed level set of the Hamiltonian $f$.  He showed that for the first occurring singularities of such pairs, which he called ``glancing hypersurfaces", there exists a simple normal form in the $C^{\infty}$-category:  
\begin{equation}
\label{nf-Mel}
\omega=dx\wedge dy+\sum_{i=1}^ndp_i\wedge dq_i, \quad X_0=\{y=0\}, \quad H=\{x^2+y+p_1=0\}
\end{equation}   
(in the analytic category and for $\dim X\geq 4$ there are functional moduli  of Ecalle-Voronin type, related to the analytic classification of pairs of involutions, c.f. \cite{Os}, \cite{V}). Later he showed in \cite{M1} that this is the only simple normal form that exists (this is in fact only true for $\dim X\geq 4$). Despite this fact, V. I. Arnol'd showed in \cite{A1} that one could proceed two more steps in the classification (of asymptotic and bi-asymptotic rays in his terminology) without encountering any moduli, as long as one replaces one of the hypersurfaces, say $X_0$, by its trace $X_0\cap H\subset H$ inside $H$.   Arnol'd's discoveries led in turn to a series of many interesting results in symplectic and Riemannian geometry, in variational calculus and optimal control, in the theory of Lagrangian and Legendrian singularities (caustics and wave fronts), and many others, most of which are now well known and collected in a series of papers and books (c.f. \cite{A1}-\cite{A4} and references therein). 

Surprisingly enough, the problem of symplectic classification of Hamiltonian systems with one-sided constraints $(\omega,f,H)$ has remained open throughout all these years, except from some recently treated cases  \cite{KM}, \cite{K1} (see also \cite{Kir}, \cite{K} for the $2$-dimensional case), where the first occurring singularities, corresponding to Melrose's glancing hypersurfaces above, were treated. In particular, the main result in \cite{K1} is an exact normal form with functional invariants (functional moduli) for the first occurring singularities of sections $H$ by symplectomorphisms of the standard Darboux normal form of the Hamiltonian system $(\omega,f)$ (see also Theorem \ref{thm3} in Section \ref{sec2}). Moreover, it was shown that the corresponding moduli space is ``huge", since it properly contains the space of all symplectic structures over the reduced phase space of the fiber $X_0$ (see also Corollary \ref{cor1} in Section \ref{sec2}). The purpose of the present paper is to continue the classification for all further typical singularity classes, and obtain exact normal forms with functional invariants for generic sections $H$ of Hamiltonian systems $(X,\omega,f)$.

The structure is as follows: in Section \ref{sec2} we describe the hierarchy of typical singularities and we formulate the corresponding theorems on normal forms, as well as some corollaries concerning the structure of the associated moduli spaces. Technically, the most important theorem here is Theorem \ref{thm2} which gives a preliminary normal form for any typical singularity $H$, containing $k$-functions defined over the orbit space of the system $(\omega,f)$, which is a quasi-symplectic space (and in fact a Poisson manifold), i.e. a $(\dim X-1)$-dimensional space endowed with a closed $2$-form of maximal rank (that is $\dim X-2$). This allows us to reduce the classification problem to the problem of classification of mappings $R=(R_0,\cdots,R_{k-1})$  with prescribed singularities in quasi-symplectic space.

In Section \ref{sec3} we present all the technical details which allow us to obtain exact normal forms for such mappings. They concern the symplectic classification of mappings $r=(r_0,\cdots,r_{2n-1})$ with certain Whitney-type singularities, naturally defined on the reduced phase space of the 0-level set $X_0$ of $f$, which is the subject of Theorem \ref{thm-wk}. The results in this section provide a direct generalisation of the results in \cite{K1}, which give an exact normal form with functional invariants for non-singular mappings $r=(r_0,\cdots,r_{2n-1})$, i.e. for diffeomorphisms in symplectic space.

Finally in Section \ref{sec4} we present the proofs of Theorems \ref{thm2}-\ref{thm7} on exact normal forms, each one of them occupying only a few lines, due to the results of Sections \ref{sec2} and \ref{sec3}. 

Before we close this section we remark that there are several important geometric-invariant objects associated to a section $H$ of a Hamiltonian system $(\omega,f)$ which are not mentioned in the present exposition. Following \cite{A11}, one may think for example the discriminant in the orbit space of the system, the (singular) Lagrangian manifolds of all solutions on them, e.t.c., all objects defining $1$-parameter families with respect to the values $f=t$ of the Hamiltonian, of the corresponding objects associated to the pair of hypersurfaces $(H,X_0)$. In fact, the underlying geometric theory of such an approach, is (in the holomorphic category) a symplectic version of the standard theory of the discriminant and the period map (monodromy, Gauss-Manin connections, mixed Hodge structures e.t.c.) for the series $A_k$ of simple boundary singularities of functions in Arnol'd's list \cite{A5}, already described in \cite{K} for the $2$-dimensional case. Indeed, as the experienced reader may verify, the functional invariants obtained in the present paper can be naturally related with the time functions, or equivalently, the action/period integrals of the (Liouville $1$-form $a$ of the) symplectic form $\omega=da$ along the ``vanishing half-cycles" (in the terminology of V. I. Arnol'd \cite{A5}), formed by the segments of the integral curves of the Hamiltonian vector field $Z_f$ between two consecutive  points of intersection with the hypersurface $H$, away from the discriminant. This geometric/cohomological description of the functional invariants announced here, as well as other more profound relations with the asymptotic expansion of these period integrals as one approaches the discriminant of the system (spectrum, asymptotic Hodge filtration e.t.c.), will be analysed in a subsequent paper.   


\section{Hierarchy of Singularities-Theorems on Normal Forms}
\label{sec2}

All the objects here are $C^{\infty}$ or analytic germs at the origin of $(X,0):=(\R^{2n+2},0)$, $n\geq 1$, unless otherwise stated. The results hold in the complex analytic (holomorphic) category in $(X,0):=(\C^{2n+2},0)$ as well, by considering diffeomorphisms (biholomorphisms) tangent to the identity. For the planar case $n=0$ we refer to \cite{K} for the holomorphic case and also to \cite{Kir}, \cite{KM}, \cite{K1} for the smooth case.

We describe here the hierarchy of singularities (partially presented also in \cite{A1}, \cite{M}, \cite{M1}) and we formulate theorems on normal forms for generic sections $H\subset \R^{2n+2}$ of a fixed regular (non-singular) Hamiltonian system $(\omega,f)$, $f(0)=0$, $df(0)\ne 0$: two sections $H=\{h=0\}$, $H'=\{h'=0\}$ will be called equivalent if there exists a symplectomorphism $\Phi$ of $\omega$, $\Phi^*\omega=\omega$, preserving also $f$, $\Phi^*f=f$, and an invertible function $u$ (a unit, $u(0)\ne 0$) such that:
\[\Phi^*h'=uh.\]
By the symplectic rectification theorem we may always fix the standard normal form of the pair $(\omega,f)$:
\begin{equation}
\label{nfHS}
\omega=dx\wedge dy+dp\wedge dq, \quad f=y,
\end{equation}
(where we denote by $dp\wedge dq:=\sum_{i=1}^ndp_i\wedge dq_i$ the standard symplectic form in $\mathbb{R}^{2n}$) and classify sections $H=\{h=0\}$ under the isotropy subgroup of this normal form.
 
\medskip

\noindent NOTATION: Throughout the paper we identify the hypersurface $H$ with its defining equation $H=\{h=0\}$ where $h(0)=0$, $dh(0)\ne 0$ (the definitions below are independent of the choice of $h$). We also denote by $X_0=\{f=0\}$ the hypersurface defined by the $0$-level set of $f$.

\medskip

\begin{definition}
\label{def1}
\it{We say that a section $H$ of a Hamiltonian system $(\omega,f)$ is non-singular if:
\[\{f,h\}(0)\ne 0.\]}
\end{definition}
The condition implies that the hypersurface $H$ is transversal to the Hamiltonian vector field $Z_f$ of $f$, the latter defined by standard symplectic duality:
\[Z_f\lrcorner \omega=df\Longleftrightarrow Z_f=\{f,\cdot\}.\]

Classification of non-singular sections is an easy exercise (using for example the Darboux-Givental theorem or any of its variations):

\begin{theorem}[\cite{A1}, \cite{M} and also \cite{KM}, \cite{K1}]
\label{thm1}
Any non-singular section $H$ of the Hamiltonian system $(\omega,f)=(\ref{nfHS})$ is equivalent to the normal form:
\begin{equation}
\label{nfS0}
H=\{x=0\}.
\end{equation} 
\end{theorem}
As it turns out these are the only simple singularities of sections (with no moduli)  that exist. Away from the non-singular case, functional moduli appear already for the first occurring singularities.  The hierarchy is as follows:

\subsection{Hierarchy of Singularities}

For economy in the exposition we will use, whenever is necessary, the following:

\medskip 

\noindent NOTATION: For a pair of functions $(f,h)$ we denote by $\{f,h\}_i$, $i\geq 0$, (resp. $\{h,f\}_i$) the $i$-times iterated Poisson brackets of $f$ with $\{f,h\}$ (and of $h$ with $\{h,f\}=-\{f,h\}$ respectively), i.e. under the following rule:
\[\{f,h\}_i:=Z_f^{i+1}(h), \quad \{h,f\}_i:=Z_h^{i+1}(f), \quad i\geq 0,\] 
so that:
\[\{f,h\}_0:=Z_f(h):=\{f,h\}, \quad \{f,h\}_1:=Z_f^2(h):=\{f,\{f,h\}\}, \quad  \cdots,\]  
\[(\{h,f\}_0:=Z_h(f):=\{h,f\}, \quad \{h,f\}_1:=Z_h^2(f):=\{h,\{h,f\}\}, \quad \cdots, \quad resp.)\]

\medskip 

\begin{definition}
\label{def2}
\it{A section $H$ of a Hamiltonian system $(\omega,f)$ is said to have a singularity of type $S_k$, $k\geq 1$, if the following conditions hold:
\begin{equation}
\label{ctr}
df\wedge dh(0)\ne 0,
\end{equation}
and
\begin{equation}
\label{cSk}
\{f,h\}_i(0)=0, \quad \forall i=0,\cdots,k-1,\quad \{f,h\}_{k}(0)\ne 0.
\end{equation}}
\end{definition}

In the definition above condition (\ref{ctr}) means that the hypersurface $X_0$ is transversal to $H$. Conditions (\ref{cSk}) imply that the Hamiltonian vector field $Z_f$ of $f$ has exactly $k$-order tangency with the hypersurface $H$. Thus, the singularities $S_k$ are of the standard Whitney type: the first occurring singularity class $S_1$ (fold singularities) is defined by $\{f,h\}(0)=0$, $\{f,\{f,h\}\}(0)\ne 0$, the next singularity class $S_2$ (cusp singularities) by $\{f,h\}(0)=0$, $\{f,\{f,h\}\}(0)= 0$, $\{f,\{f,\{f,h\}\}\}(0)\ne 0$ and so on. 

\begin{remark}
\label{rem1}
\normalfont{It is obvious from the definition that the set of points where a section $H$ has singularity $S_k$ forms a subset of codimension $k$ in $H$. Thus, the only typical singularity classes $S_k$ in $\R^{2n+2}$ are those that appear in the range $1\leq k \leq 2n+1$.}
\end{remark}

Each of the singularity classes $S_k$ is naturally stratified further to singularity subclasses defined by the relative positions of the function $f$ (of the hypersurface $X_0$) with respect to the characteristic line field of the hypersurface $H$, i.e. the one obtained by the 1-dimensional field of kernels of the restriction $\omega_H$ of the symplectic form $\omega$ on $H$. In that way a double index series of singularities arises: 

\begin{definition}
\label{def3}
A section $H\in S_k$ has a singularity of type $S_{k,l}$, $l\geq 1$, if the following further conditions hold:
\begin{equation}
\label{cSkl}
\{h,f\}_i(0)=0,\quad \forall i=1,\cdots,l-1, \quad \{h,f\}_{l}(0)\ne0.
\end{equation}
\end{definition}

As before, the conditions (\ref{cSkl}) are equivalent to the conditions that the Hamiltonian vector field $Z_H=Z_h|_H$ spanning the characteristic line field of $H$, has exactly $l$-order tangency with the hypersurface $X_0$.  Thus, the singularities $S_{k,l}$ are again of the standard Whitney-Whitney type: the first occurring singularity $S_{1,1}$ (fold-fold) is defined by:
\[\{f,h\}(0)=0, \quad \{f,\{f,h\}\}(0)\ne 0, \quad \{h,\{h,f\}\}(0)\ne 0,\]
whereas adjacent to it are the singularities $S_{1,2}$ (fold-cusp):
\[\{f,h\}(0)=\{h,\{h,f\}\}(0)=0, \quad \{h,\{h,\{h,f\}\}\}(0)\ne 0, \quad \{f,\{f,h\}\}(0)\ne 0,\]
as well as the singularities $S_{2,1}$ (cusp-fold):
\[\{f,h\}(0)=\{f,\{f,h\}\}(0)=0, \quad \{f,\{f,\{f,h\}\}\}(0)\ne 0, \quad \{h,\{h,f\}\}(0)\ne 0,\]
and so on.

\begin{remark}
\label{rem2}
\normalfont{As it is obvious again from the definition, the set of points where a section $H$ has singularity of type $S_{k,l}$ forms a subset of codimension $k+l-1$ in $H$. Thus, the typical singularity classes $S_{k,l}$ in $\R^{2n+2}$ are only those that appear in the range $1\leq k+l-1\leq 2n+1$.}
\end{remark}

Away from the singularites $S_{k,l}$ described above, there is one more typical singularity class, adjacent to the first occurring singularities $S_{1,1}$, which is defined by the loss of transversality of $f$ (equivalently, of $X_0$) with $H$, i.e. such that condition (\ref{ctr}) in Definition \ref{def2} ceases to hold.

\begin{definition}
\label{def4}
\it{We say that a section $H$ has a singularity of type $A_1$ if the restriction $f|_H$ of $f$ on $H$ has a non-degenerate critical point at the origin (i.e. it is a Morse function):
\begin{equation}
\label{cA1}
df|_{H}(0)=0, \quad d^2f|_{H}(0)\ne 0,
\end{equation}
and moreover the following conditions hold:
\begin{equation}
\label{cA11}
\{f,\{f,h\}\}(0)\ne 0, \quad \{h,\{h,f\}\}(0)\ne 0.
\end{equation}}
\end{definition}

\begin{remark}
\label{remA1}
\normalfont{The condition that the restriction $f|_H$ is a Morse function is obviously equivalent to the condition that the restriction $h|_{X_0}$ of the function $h$ defining $H=\{h=0\}$ on the $0$-level set $X_0=\{f=0\}$ of $f$, is also a Morse function, i.e. with a non-degenerate critical point at the origin:
\[dh|_{X_0}(0)=0,\quad d^2h|_{X_0}(0)\ne 0.\]
This, along with conditions (\ref{cA11}), gives an equivalent definition for $A_1$-singularities, more adequate for our purposes (see Theorem \ref{thm7} below).} 
\end{remark}

\subsection{Normal Forms and Functional Moduli}

Here we state the corresponding theorems on exact normal forms with functional invariants for all the typical singularities of sections $S_{k.l}$ and $A_1$ defined above. We give first a preliminary normal form (i.e. containing non-equivalent germs) which is in a sense the ``simplest" normal form, serving for the whole singularity class $S_k$, for each $k\geq 1$ (as well as for $A_1$). As it will become apparent in the text, having such a normal form is a decisive step in order to obtain further classification results for all singularity subclasses $S_{k,l}$. 

\begin{theorem}
\label{thm2}
Any section $H\in S_{k}$ of the Hamiltonian system $(\omega,f)=(\ref{nfHS})$ is equivalent to the normal form:
\begin{equation}
\label{nfSk}
H=\{x^{k+1}+\sum_{i=0}^{k-1}R_{i}(y,p,q)x^{i}=0\},
\end{equation}
for some functions $R_{i}(y,p,q)$, $i=0,\cdots,k-1$, such that $R_{i}(0)=0$, $i=0,\cdots,k-1$, and:
\begin{equation}
\label{ctr'}
dy\wedge dR_0(0)\ne 0,
\end{equation}
\begin{equation}
\label{cS11R'}
H\in S_{1,1}: \quad \partial_yR_0(0)\ne 0, 
\end{equation}  
\[H\in S_{1, l\geq 2}: \quad \partial_yR_0(0)=\{R_0,\partial_yR_0\}_i(0)=0, \quad i=0,\cdots,l-3,\]
\begin{equation}
\label{cS1lR'}
 \{R_0,\partial_yR_0\}_{l-2}(0)\ne 0 
\end{equation}  
\begin{equation}
\label{cSklR'}
H\in S_{k\geq 2,l}: \quad \{R_0,R_1\}_i(0)=0, \quad i=0,\cdots,l-2,  \quad \{R_0,R_1\}_{l-1}(0)\ne 0. 
\end{equation}  
\end{theorem}

\begin{remark}
\label{rem0}
\normalfont{Condition (\ref{ctr'}) on the function $R_0$ is just the transversality condition (\ref{ctr}) in Definition \ref{def2}, in terms of the normal form (\ref{nfSk}). Condition (\ref{cS11R'}) follows by definition $\{h,\{h,f\}(0)\ne 0$. Condition (\ref{cS1lR'}) follows by condition (\ref{cSkl}) in Definition \ref{def3} and a simple calculation:
 \[\{h,f\}_1(0)=\partial_yR_0(0),\quad \{h,f\}_i(0)=\{R_0,\partial_yR_0\}_{i-2}(0), \quad \forall i\geq 2,\]
whereas condition (\ref{cSklR'}) follows also by condition (\ref{cSkl}) in Definition \ref{def3} where now, as one may easily verify:
\[\{h,f\}_i(0)=\{R_0,R_1\}_{i-1}(0), \quad \forall i\geq 1.\]}
\end{remark}

Theorem \ref{thm2} above allows us to reduce the classification problem of sections $H\in S_{k,l}$ to the problem of classification of the associated coefficient functions (the coefficient mapping) $R(y,p,q)=(R_0(y,p,q),\cdots,R_{k-1}(y,p,q))$ in normal form (\ref{nfSk}), satisfying conditions (\ref{ctr'})-(\ref{cSklR'}), under (quasi-)symplectomorphisms of the pair $(y,dp\wedge dq)$, i.e. symplectomorphisms of the form:
\[(y,p,q)\mapsto (y,\Phi(p,q)), \quad \Phi^{*}(dp\wedge dq) =dp\wedge dq.\]
Indeed, any symplectomorphism of the pair $(\omega,f)=(\ref{nfHS})$ which also preserves the "$k$-singular locus" of $H$:
\[C_k:=\{\{f,h\}_{k-1}=0\}=\{x=0\},\]
is necessarily of the form:
\[(x,y,p,q)\mapsto (x,y,\Phi(p,q)), \quad \Phi^{*}(dp\wedge dq) =dp\wedge dq.\]

In order to obtain exact normal forms for such mappings, notice that conditions (\ref{ctr'})-(\ref{cSklR'}) in Theorem \ref{thm2} define conditions only on the first pair of functions $(R_0,R_1)$, and in fact, on their restrictions $(r_0(p,q),r_1(p,q)):=(R_0(0,p,q),R_1(0,p,q))$ on the reduced symplectic space $(\R^{2n}_{(p,q)},dp\wedge dq)$ of the $0$-fiber $X_0=\{y=0\}$ of $f=y$, whereas the rest of the functions $R_i$, $i=2,\cdots,k-1$, are arbitrary. Thus they can be supposed to satisfy some generic condition with respect to the initial pair $(R_0,R_1)$, in the sense that the corresponding jet of the coefficient map $R=(R_0,\cdots,R_{k-1})$ belongs in an open set in the appropriate space of jets of mappings in quasi-symplectic space $(\R^{2n+1}_{(y,p,q)}, dp\wedge dq)\rightarrow \R^k$, satisfying conditions (\ref{ctr'})-(\ref{cSklR'}) above. Of course, the choice is not unique and different choices lead to different exact normal forms with functional moduli, given possibly by a different number of functions in a different number of variables. Below we present the corresponding results under certain natural generic conditions on the Taylor expansions of the functions $R_i(y,p,q)$, $i=0,\cdots,k-1$, along the variable $y=f$:
\[R_i(y,p,q)=\sum_{j\geq 0}r_{i,j}(p,q)y^j, \quad i=0,\cdots,k-1.\]
For economy in the exposition we will use throughout the paper the following: 

\medskip 

\noindent NOTATION:  We denote by $\I_i$, $i=1,\cdots, 2n$, the nested sequence of ideals generated by the first $i$-Darboux coordinate functions of the Darboux normal form $\widetilde{\omega}=dp\wedge dq(=\sum_{i=1}^{n}dp_i\wedge dq_i)$ of a symplectic structure in $\R^{2n}$, but in reverse order, i.e.:
\[\I_1:=<q_1>,\]
\[\I_2:=<q_1,p_1>,\]
\[\vdots\]
\[\I_{2n-2}:=<q_1,p_1,\cdots,q_{n-1},p_{n-1}>,\]
\[\I_{2n-1}:=<q_1,p_1,\cdots,q_{n-1},p_{n-1},q_n>,\]
\[\I_{2n}:=<q_1,p_1,\cdots,q_{n-1},p_{n-1},q_n,p_n>.\]

\subsubsection{The Series $S_{k,1}$, $k\geq 1$.}
We start first with the classification of singularities $S_{k,1}$, for $1\leq k\leq 2n$, due to the fact that the statement of the theorem for the isolated case $k=2n+1$ is different (but its proof is similar/easier).

\begin{theorem}
\label{thm3}
In the space of $(\max\{k,2n-k+1\})$-jets of sections $H\in S_{k,1}$, $1\leq k\leq 2n$, of the Hamiltonian system $(\omega,f)=(\ref{nfHS})$ there exists an open set $U$, such that any section $H$ with $j^{\max\{k,2n-k+1\}}H\in U$ is equivalent to the normal form (\ref{nfSk}), where:
\begin{equation}
\label{nfRiSk1}
R_i(y,p,q)=r_{i}(p,q)+\phi_i(y,p,q)y, \quad i=1,\cdots, k-1,
\end{equation}
\begin{equation}
\label{nfR0Sk1}
R_0(y,p,q)=g(y)+p_1+\sum_{j=1}^{2n-k}r_{k-1+j}(p,q)y^j+\phi_0(y,p,q)y^{2n-k+1},
\end{equation}
with: 
\begin{equation}
\label{nfgSk1}
g'(0)\ne 0,
\end{equation} 
and the functions $r_i$, $i=1,\cdots, 2n-1$ are of the form:
\begin{equation}
\label{nfroddSk1}
r_{2m+1}\in \I_{2m+1}, \quad \partial_{q_{m+1}}r_{2m+1}(0)\ne 0, \quad m=0,\cdots,n-1,
\end{equation}
\begin{equation}
\label{nfrevenSk1}
r_{2m}(p,q)=p_{m+1}+\widetilde{r}_{2m}(p,q), \quad \widetilde{r}_{2m}\in \I_{2m}, \quad m=1,\cdots,n-1.
\end{equation}
The set of functions:
\begin{equation}
\label{mk1}
\text{m}_{k,1}(H)=\big \{\{\phi_i\}_{i=0}^{k-1},g,\{r_{2m+1}\in \I_{2m+1}\}_{m=0}^{n-1},\{\widetilde{r}_{2m}\in \I_{2m}\}_{m=1}^{n-1} \big \}
\end{equation}
is a complete set of functional moduli for the classification of generic sections $H\in S_{k,1}$, $1\leq k \leq 2n$: two sections $H$, $H'$ of $(\omega,f)=(\ref{nfHS})$ are equivalent if and only if $\text{m}_{k,1}(H)\equiv \text{m}_{k,1}(H')$.
\end{theorem}

\begin{example}
\label{ex1}
\normalfont{The normal form for the first occurring singularities $S_{1,1}$ has been recently obtained in \cite{K1}:
\[H=\{x^2+R_0(y,p,q)=0\},\]
where:
\[R_0(y,p,q)=g(y)+\sum_{m=0}^{n-1}r_{2m+1}(p,q)y^{2m+1}+\sum_{m=0}^{n-1}(p_{m+1}+\widetilde{r}_{2m}(p,q))y^{2m}+\phi_0(y,p,q)y^{2n},\]
(where $\widetilde{r}_0\equiv 0$) with the set of functional invariants:
\[\text{m}_{1,1}(H)=\big \{\phi_0,g,\{r_{2m+1}\in \I_{2m+1}\}_{m=0}^{n-1},\{\widetilde{r}_{2m}\in \I_{2m}\}_{m=1}^{n-1} \big \},\]
where $g'(0)\ne 0$. A different exact normal form has also been obtained in \cite{KM} (see again \cite{K1} for the diffeomorphism sending one normal form to the other). }
\end{example}

\begin{remark}
\label{rem3}
\normalfont{For the $S_{1,1}$ case, the condition (\ref{nfgSk1}) that the functional invariant $g(y)$ has non-vanishing $1$-jet, $g'(0)\ne 0$, is implied by the condition $\{h,\{h,f\}\}(0)\ne 0$. For the further singularity classes $S_{k,1}$, $k\geq 2$, it is not implied by any condition, but instead it is a natural generic condition related to the ``transversal type" $H\cap\{p=q=0\})$ of $H\in S_{k,1}$: the intersection of the section $H$ with the $2$-dimensional symplectic subspace $(\R^2_{(x,y)},dx\wedge dy)$ defined by the $0$-fiber of the double projection $\pi(x,y,p,q)=(p,q)$ (from the total space $\R^{2n+2}_{(x,y,p,q)}$ to the orbit space $\R^{2n+1}_{(y,p,q)}$ of $Z_f=\partial_x$ and then to the reduced phase space $\R^{2n}_{(p,q)}$ of the fiber $X_0$ of $f=y$), is a smooth plane curve, or equivalently $dh\wedge (dp\wedge dq)^n(0)\ne 0$.
}
\end{remark}

As it will become apparent in the proof of the theorem (see Section \ref{sec4:2}), the open set $U$ in the space of jets of sections $H\in S_{k,1}$ appearing in its statement, is defined by the condition that the associated map $r(p,q)=(p_1,r_1(p,q),\cdots,r_{2n-1}(p,q))$ obtained by the first $2n$-coefficients in the Taylor expansions along $f=y$ of the functions $R_i(y,p,q)$, $i=0,\cdots,k-1$, given by (\ref{nfRiSk1})-(\ref{nfR0Sk1}), defines a diffeomorphism in the reduced phase space $(\R^{2n}_{(p,q)},dp\wedge dq)$ of the $0$-level set $X_0=\{y=0\}$ of $f=y$. From the results of Section \ref{sec3} (Theorem \ref{thm-wk}), it follows that the expressions (\ref{nfroddSk1})-(\ref{nfrevenSk1}) provide the exact normal form for the diffeomorphism $r\in \text{Diff}(\R^{2n})$, under symplectomorphisms $\text{Symp}(\widetilde{\omega})$ of the Darboux normal form $\widetilde{\omega}=dp\wedge dq$. Bringing back the mapping $r$ to the identity, $r(p,q)=(p,q)$ we immediately see that the new symplectic form $(r^{-1})^*\widetilde{\omega}$ is a functional invariant, and thus the invariant ideals $\I_i$, $i=1,\cdots,2n-1$ appearing in normal form (\ref{nfroddSk1})-(\ref{nfrevenSk1}), give a natural parametrisation of the space $\Omega^2_S(\R^{2n})$ of all symplectic structures in $\R^{2n}$, which in that way becomes an ``infinite dimensional homogeneous space":
\[\frac{\text{Diff}(\R^{2n})}{\text{Symp}(\widetilde{\omega})}\cong \Omega^2_S(\R^{2n})\cong C^{\infty}(\R^{2n})^{n(2n-1)},\]
with the last isomorphism obtained by expanding each of the functions in the corresponding ideal $\I_i$, as a sum over its generators  (c.f. \cite{K1} for details). From this it follows:

\begin{corollary}
\label{cor1}
The moduli space $\mathcal{M}_{k,1}(U)$ of generic sections $H\in S_{k,1}$, $1\leq k\leq 2n$, is isomorphic to the product:
\[\mathcal{M}_{k,1}(U)\cong \Omega^2_S(\R^{2n})\times \text{Diff}(\R)\times C^{\infty}(\R^{2n+1})^k.\]
In particular,
\[\mathcal{M}_{2n,1}(U)\cong \Omega^2_{QS}(\R^{2n+1}),\]
is the space of all quasi-symplectic structures in $\R^{2n+1}$.
\end{corollary} 
\begin{proof}
The first isomorphism follows immediately from (\ref{mk1}) which shows that starting from $k=1$, the functional moduli $m_{k,1}(H)$ grow with $k$ by adding $k$ functions of $(2n+1)$-variables (the functions $\{\phi_i\}_{i=1}^k$ in (\ref{mk1})). For the second isomorphism notice that the adjoint map:
\[(y,p,q)\mapsto (y,R_0(y,p,q),\cdots,R_{2n-1}(y,p,q))\]
also defines a diffeomorphism in the quasi-symplectic space $(\mathbb{R}^{2n+1}_{(y,p,q)},dp\wedge dq)$. If we denote by $\text{QSymp}(\widehat{\omega})$ the group of quasi-symplectomorphisms of $\widehat{\omega}=dp\wedge dq$ in $\R^{2n+1}_{(y,p,q)}$, we conclude by the same argument as above the isomorphism:
\[\frac{\text{Diff}(\R^{2n+1})}{\text{QSymp}(\widehat{\omega})}\cong \Omega^2_{QS}(\R^{2n+1}).\] 
\end{proof}

\begin{remark}
\label{rem4}
\normalfont{The moduli spaces described above depend a priori on the choice of the open set $U$ and in particular on the choice of the normal form. Different normal forms might lead to different moduli spaces and in these terms, the problem to obtain a canonical description (independent of coordinate choices) of the corresponding moduli spaces remains open. This would be useful, among others, in order to calculate the Poincar\'e series of the classification problem, a problem in the circle of problems proposed by Arnol'd, c.f. \cite{A6}.}
\end{remark}

As far as it concerns the classification of isolated singularities $S_{2n+1,1}$ we obtain by the same method:
\begin{theorem}
\label{thm4}
Any generic section $H\in S_{2n+1,1}$ of the Hamiltonian system $(\omega,f)=(\ref{nfHS})$ is equivalent to the normal form (\ref{nfSk}), where $\partial_{y}R_{2n}(0)\ne 0$ and the functions $R_i(y,p,q)$, $i=0,\cdots,2n-1$, are of the form:
\begin{equation}
\label{nfRiS2n+11}
R_i(y,p,q)=r_{i}(p,q)+\phi_i(y,p,q)y, \quad i=0,\cdots, 2n-1,
\end{equation}
where $\phi_0(0)\ne 0$ and: 
\begin{equation}
\label{nfr02n+1}
r_0(p,q)=p_1
\end{equation}
\begin{equation}
\label{nfroddS2n+11}
r_{2m+1}\in \I_{2m+1}, \quad \partial_{q_{m+1}}r_{2m+1}(0)\ne 0, \quad m=0,\cdots,n-1,
\end{equation}
\begin{equation}
\label{nfrevenS2n+11}
r_{2m}(p,q)=p_{m+1}+\widetilde{r}_{2m}(p,q), \quad \widetilde{r}_{2m}\in \I_{2m}, \quad m=1,\cdots,n-1.
\end{equation}
The set of functions:
\[\text{m}_{2n+1,1}(H)=\big \{R_{2n}, \{\phi_i\}_{i=0}^{2n-1},\{r_{2m+1}\in \I_{2m+1}\}_{m=0}^{n-1},\{\widetilde{r}_{2m}\in \I_{2m}\}_{m=1}^{n-1} \big \}\]
is a complete set of functional moduli for the classification of generic sections $H\in S_{2n+1,1}$: two sections $H$, $H'$ of $(\omega,f)=(\ref{nfHS})$ are equivalent if and only if $\text{m}_{2n+1,1}(H)\equiv \text{m}_{2n+1,1}(H')$.
\end{theorem}

\begin{remark}
\normalfont{The condition that the function $\phi_0$ is non-vanishing at the origin $\phi_0(0)\ne 0$, is a generic condition related again to the transversal type $H\cap \{p=q=0\}$ of $H$, exactly as in Remark \ref{rem3}.}
\end{remark}

The genericity condition stated in the theorem is simply the functional independence of the $(2n+1)$-functions $R_i(y,p,q)$ at the origin of $\R^{2n+1}_{(y,p,q)}$:
\[dR_0\wedge \cdots \wedge dR_{2n}(0)\ne 0.\]
It is easy to see now that the moduli space $\mathcal{M}_{2n+1,1}$ is obtained by the space $\Omega^2_{QS}(\R^{2n+1})$ of all quasi-symplectic structures in $\R^{2n+1}$, by adding one function of $(2n+1)$-variables (the function $R_{2n}(y,p,q)$ in the statement of the theorem). Indeed, the adjoint map:
\[(y,p,q)\mapsto (y,R_0(y,p,q),\cdots,R_{2n-1}(y,p,q)),\]
is by assumption non-singular (i.e. $dy\wedge dR_0\wedge \cdots \wedge dR_{2n-1}(0)\ne 0$) and thus defines a diffeomorphism in the quasi-symplectic space $(\R^{2n+1}_{(y,p,q)}, dp\wedge dq)$. Thus, by Corollary \ref{cor1} above we obtain:

\begin{corollary}
\label{cor2}
The moduli space $\mathcal{M}_{2n+1,1}$ of generic sections $H\in S_{2n+1,1}$, is isomorphic to the product: 
\[\mathcal{M}_{2n+1,1}\cong \Omega^2_{QS}(\R^{2n})\times C^{\infty}(\R^{2n+1}).\]
\end{corollary}
 
\subsubsection{The Series $S_{1,l}$, $l\geq 2$}
Below we consider the single index series $S_{1,l}$, $l\geq 2$. 

\medskip 

\noindent NOTATION: We denote throughout the rest of the paper by $r(\widehat{\cdot})$ any function $r$ not depending on the variables under the ``hat'' symbol.

\medskip 

\begin{theorem}
\label{thm5}
In the space of $(\max\{l,2n\})$-jets of sections $H\in S_{1,l}$, $2\leq l\leq 2n+1$, of the Hamiltonian system $(\omega,f)=(\ref{nfHS})$ there exists an open set $U$, such that any section $H$ with $j^{\max\{l,2n\}}H\in U$ is equivalent to the normal form (\ref{nfSk}) of Theorem \ref{thm2} for $k=1$, where the function $R_0(y,p,q)$ is of the form:
\begin{equation}
\label{nfR0S1l}
R_0(y,p,q)=g(y)+p_1+\sum_{j=1}^{2n-1}r_{j}(p,q)y^j+\phi_0(y,p,q)y^{2n},
\end{equation}
where: 
\begin{eqnarray}
\label{cnfR0S1l}
g'(0)=0, \quad g''(0)\ne 0,
\end{eqnarray}
and:
\begin{equation}
\label{nfr1S1l}
r_1(p,q)=\psi(p,q)q_1^{l-1}+\sum_{j=0}^{l-3}r_{1,j}(\widehat{q}_1)q_1^{j},
\end{equation}
where $\psi(0)\ne 0$, $r_{1,j}(0)=0$, $j=0,\cdots,l-3$, and the functions $r_{1,j}$ are differentially independent at the origin:
\begin{equation}
\label{dir1j}
dr_{1,0}\wedge \cdots \wedge dr_{1,l-2}(0)\ne 0,
\end{equation}
whereas the rest of the functions $r_i$, $i=2,\cdots,n-1$ are of the same form as in Theorems \ref{thm3}-\ref{thm4}:
\begin{equation}
\label{nfroddS1l}
r_{2m+1}\in \I_{2m+1}, \quad \partial_{q_{m+1}}r_{2m+1}(0)\ne 0, \quad m=1,\cdots,n-1,
\end{equation}
\begin{equation}
\label{nfrevenS1l}
r_{2m}(p,q)=p_{m+1}+\widetilde{r}_{2m}(p,q), \quad \widetilde{r}_{2m}\in \I_{2m}, \quad m=1,\cdots,n-1,
\end{equation}
The set of functions:
\begin{equation}
\label{m1l}
\text{m}_{1,l}(H)=\big \{\phi_0,g,\psi,\{r_{1,j}\}_{j=0}^{l-3},\{r_{2m+1}\in \I_{2m+1}\}_{m=1}^{n-1},\{\widetilde{r}_{2m}\in \I_{2m}\}_{m=1}^{n-1} \big \}
\end{equation}
is a complete set of functional moduli for the classification of generic sections $H\in S_{1,l}$: two sections $H$, $H'$ of $(\omega,f)=(\ref{nfHS})$ are equivalent if and only if $\text{m}_{1,l}(H')\equiv \text{m}_{1,l}(H')$.
\end{theorem}

\begin{remark}
\label{rem5}
\normalfont{The condition (\ref{cnfR0S1l}) that the function of $1$-variable $g(y)$ has a non-degenerate critical point at the origin (is a Morse function) is obtained by definition $\{h,\{f,h\}(0)=0$, and the generic assumption that the transversal type $H\cap \{p=q=0\}$ of $H$ (defined as in Remark \ref{rem3}) is a singular plane curve with quadratic singularity (an $A_1$-singularity) at the origin. }
\end{remark}

As it will become apparent in the course of the proof of the theorem, the open set $U$ in the space of jets of sections $H\in S_{1,l}$ appearing in the statement, is defined by the condition that the associated mapping $r(p,q)=(p_1,r_1(p,q),\cdots,r_{2n-1}(p,q))$, obtained by the first $2n$ coefficients of the Taylor expansion of the function $R_0(y,p,q)$ along $f=y$, given by (\ref{nfR0S1l}), is a diffeomorphism for $l=2$, and defines a mapping with certain singularities of ``Whitney-type" (see Section \ref{sec3}) in the reduced symplectic space $(\R^{2n},dp\wedge dq)$. It follows from this (Theorem \ref{thm-wk} again) that the expressions (\ref{nfr1S1l})-(\ref{nfrevenS1l}) provide the exact normal form for the coefficient mapping $r(p,q)$, under symplectomorphisms of the Darboux normal form $\widetilde{\omega}=dp\wedge dq$. From this it follows: 

\begin{corollary}
\label{cor3}
The moduli space $\mathcal{M}_{1,l}(U)$ of generic sections $H\in S_{1,l}$, $l\geq 2$, is isomorphic to the product:
\[\mathcal{M}_{1,l}(U)\cong \Omega^2_S(\R^{2n})\times C^{\infty}(\R^{2n-1})^{l-2}\times \text{M}(\R)\times C^{\infty}(\R^{2n+1}), \]
where $\text{M}(\R)$ is the set of Morse functions in one variable. 
\end{corollary}
\begin{proof}
Starting with $l=2$, where the map $r(p,q)=(p_1,r_1(p,q),\cdots,r_{2n-1}(p,q))$ defines a diffeomorphism in symplectic space $(\R^{2n},dp\wedge dq)$, we see from Theorem \ref{thm5} that the functional moduli grow with $l\geq 2$ by adding $(l-2)$-functions of $(2n-1)$-variables (the functions $r_{1j}(\widehat{q}_1)$ in (\ref{nfr1S1l})).  
\end{proof}

\begin{remark}
\normalfont{As we will see in Section \ref{sec3} below, the $(l-2)$-functional invariants of $(2n-1)$-variables $\{r_{1j}(\widehat{q}_1)\}_{j=0}^{l-3}$, are already invariants for the classification of ordinary Whitney mappings $r\in S_{l-2}$, (c.f. \cite{A11}, \cite{GG}) under ordinary $\mathcal{R}$-equivalence, i.e (non-symplectic) diffeomorpshisms of the source $\R^{2n}$ only (no changes of coordinates in the target are allowed, see Remark \ref{rem2Ss} in Section \ref{sec3}.)} 
\end{remark}

\subsubsection{The Series $S_{k,l}$, $k\geq2$, $l\geq 2$}
Below we consider all remaining cases of typical singularities of sections $S_{k,l}$, $k\geq 2$, $l\geq 2$.  The statement of the theorem is slightly more complicated, but its proof is the same with the previous ones:

\begin{theorem}
\label{thm6}
In the space of $(\max\{k,l+1,2n-k+1\})$-jets of sections $H\in S_{k,l}$, $1\leq k+l-1\leq 2n+1$, $2\leq k\leq 2n$, $2\leq l\leq 2n$, of the Hamiltonian system $(\omega,f)=(\ref{nfHS})$ there exists an open set $U$, such that any section $H$ with $j^{\max\{k,l+1,2n-k+1\}}H\in U$ is equivalent to the normal form (\ref{nfSk}) of Theorem \ref{thm2}, where the functions $R_i(y,p,q)$, $i=0,\cdots,k-1$, are again of the same form of Theorem \ref{thm3}:
\begin{equation}
\label{nfRiSkl}
R_i(y,p,q)=r_{i}(p,q)+\phi_i(y,p,q)y, \quad i=1,\cdots, k-1,
\end{equation}
\begin{equation}
\label{nfR0Skl}
R_0(y,p,q)=g(y)+p_1+\sum_{j=1}^{2n-k}r_{k-1+j}(p,q)y^j+\phi_0(y,p,q)y^{2n-k+1},
\end{equation}
but now: 
\begin{eqnarray}
\label{cnfR0Skl}
g'(0)\ne 0,
\end{eqnarray}
and:
\begin{equation}
\label{nfr1Skl}
r_1(p,q)=\psi(p,q)q_1^{l}+\sum_{j=0}^{l-2}r_{1,j}(\widehat{q}_1)q_1^{j},
\end{equation}
where $\psi(0)\ne 0$, $r_{1,j}(0)=0$, $j=0,\cdots,l-2$, and the functions $r_{1,j}$ are again differentially independent at the origin:
\begin{equation}
\label{dir1j}
dr_{1,0}\wedge \cdots \wedge dr_{1,s-1}(0)\ne 0,
\end{equation}
whereas the rest of the functions $r_i$, $i=2,\cdots,n-1$ are of the same form as in Theorems \ref{thm3}-\ref{thm5}:
\begin{equation}
\label{nfroddSkl}
r_{2m+1}\in \I_{2m+1}, \quad \partial_{q_{m+1}}r_{2m+1}(0)\ne 0, \quad m=1,\cdots,n-1,
\end{equation}
\begin{equation}
\label{nfrevenSkl}
r_{2m}(p,q)=p_{m+1}+\widetilde{r}_{2m}(p,q), \quad \widetilde{r}_{2m}\in \I_{2m}, \quad m=1,\cdots,n-1,
\end{equation}
The set of functions:
\begin{equation}
\label{mkl}
\text{m}_{k,l}(H)=\big \{\{\phi_i\}_{i=0}^{k-1},g,\psi,\{r_{1,j}\}_{j=0}^{l-3},\{r_{2m+1}\in \I_{2m+1}\}_{m=1}^{n-1},\{\widetilde{r}_{2m}\in \I_{2m}\}_{m=1}^{n-1} \big \}
\end{equation}
is a complete set of functional moduli for the classification of generic sections $H\in S_{k,l}$: two sections $H$, $H'$ of $(\omega,f)=(\ref{nfHS})$ are equivalent if and only if $\text{m}_{k,l}(H')\equiv \text{m}_{k,l}(H')$.
\end{theorem}

\begin{remark}
\normalfont{Condition (\ref{cnfR0Skl}) is related again to the transversal type $H\cap \{p=q=0\}$ of $H$ being a smooth curve, exactly as in Remark \ref{rem3}}
\end{remark}

The open set $U$ in the space of jets of sections $H\in S_{k,l}$ appearing in the statement of the theorem, is defined again by the same condition as in the $S_{1,l}$ case, i.e.  that the associated mapping $r(p,q)=(p_1,r_1(p,q),\cdots,r_{2n-1}(p,q))$, obtained by the first $2n$ coefficients of the Taylor expansions along $f=y$ of the $k$ functions $R_i(y,p,q)$, $i=0,\cdots, k-1$, given by (\ref{nfRiSkl})-(\ref{nfR0Skl}), defines again a ``Whitney-type" mapping in the reduced symplectic space $(\R^{2n},dp\wedge dq)$, given by the exact normal form (\ref{nfr1Skl})-(\ref{nfrevenSkl}) under symplectomorphisms of $\widetilde{\omega}=dp\wedge dq$. From this it follows: 

\begin{corollary}
\label{cor4}
The moduli space $\mathcal{M}_{k,l}(U)$ of generic sections $H\in S_{k,l}$, $k\geq 2$, $l\geq 2$, is isomorphic to the product:
\[\mathcal{M}_{k,l}(U)\cong \Omega^2_S(\R^{2n})\times C^{\infty}(\R^{2n-1})^{l-1}\times \text{Diff}(\R)\times C^{\infty}(\R^{2n+1})^k.\]
\end{corollary}
\begin{proof}
Starting with the case $l=1$ of Corollary \ref{cor1}, we see that the moduli space $\mathcal{M}_{k,l}(U)$ grows with $l\geq 2$ by adding $(l-1)$-functions of $(2n-1)$-variables  corresponding to the moduli of Whitney mappings under $\mathcal{R}$-equivalence (the functions $r_{1j}(\widehat{q}_1)$ in (\ref{nfr1Skl})).  
\end{proof}

\subsubsection{$A_1$-Singularities: The Relative Birkhoff Normal Form}

Working exactly in the same way as in the case of $S_1$-singularities, it is easy to show the following: 
\begin{theorem}
\label{thm7}
Any section $H\in A_1$ of the Hamiltonian system $(\omega,f)=(\ref{nfHS})$ is equivalent to the normal form:
\begin{equation}
\label{nfA1}
H=\{x^2+\widetilde{\psi}(p,q)+\phi(y,p,q)y=0\},
\end{equation}
where  $\phi(0)\ne 0$, and the function $\widetilde{\psi}$ is an exact normal form of a Morse function $\psi$ with respect to the symplectic structure $dp\wedge dq$ in $\R^{2n}$.  
\end{theorem}

\begin{remark}
\normalfont{The theorem above reduces the problem to obtain exact normal forms for $A_1$-singularities of sections $H$, to the problem to obtain exact normal forms for Morse functions $\psi(p,q)$ in the reduced phase space $(\R^{2n},dp\wedge dq)$. This is in turn a classical subject, starting already from the works of G. D. Birkhoff \cite{B} and the so-called Birkhoff normal forms. Due to the works of several authors (c.f. more recently \cite{Zu} and references therein) it is now well known that for a generic Morse function $\psi$, reduction to its Birkhoff normal form is only possible in the formal category (unless $n=1$, as is implied for example by the isochore Morse lemma \cite{CV}, \cite{Ve} (see also \cite{F}, \cite{G}). It follows from this that for a generic section $H\in A_{1}$, reduction to its exact normal form (\ref{nfA1}) with e.g. in the elliptic case:
\[\psi=\widetilde{\psi}(p_1^2+q_1^2,\cdots,p_n^2+q_n^2),\]
is only possible in the formal category (for $n\geq 2$) as well. 	 }
\end{remark}

\section{Whitney-Type Mappings in Symplectic Space}
\label{sec3}

We consider here the $\mathcal{R}(\widetilde{\omega})$-classification of mappings $r:(\R^{2n},\widetilde{\omega})\rightarrow \R^{2n}$ having certain singularities of ``Whitney-type" (see below), by symplectomorphisms of the source space (no changes of coordinates in the target are allowed) preserving a fixed symplectic structure: 
\begin{equation}
\label{dnft}
\widetilde{\omega}=dp\wedge dq.
\end{equation}
 
\begin{definition}
\label{defrSs}
By a Whitney-type mapping $r\in S_{s}(\widetilde{\omega})$, $0\leq s\leq 2n$, in symplectic space, we mean a mapping $r:(\R^{2n},\widetilde{\omega})\rightarrow \R^{2n}$, $r=(r_0,\cdots,r_{2n-1})$ such that:
\begin{itemize}
\item[(a)] the pair of functions $(r_0,r_1)$ defines an $S_{s}$-singularity at the origin in the sense of Definitions \ref{def1}-\ref{def2}, i.e.
\[\{r_0,r_1\}(0)\ne 0,\]
for $s=0$, and:
\[\{r_0,r_1\}_i(0)=0, \quad i=0,\cdots,s-1, \quad \{r_0,r_1\}_s(0)\ne 0,\]
\[dr_0\wedge dr_1(0)\ne 0, \quad (\text{for $n\geq 2$}),\]
for $s\geq 1$,
\item[(b)] the functions $r_0$, $\{r_0,r_1\}_{i}$, $i=0,\cdots,s-1$, are differentially independent at the origin:
\[dr_0\wedge d\{r_0,r_1\}\wedge \cdots \wedge d\{r_0,r_1\}_{s-1}(0)\ne 0,\]
\item[(c)] the functions $r_0$, $\{r_0,r_1\}_{s-1}$ and $r_i$, $i=2,\cdots,2n-1$, are also differentially independent at the origin:
\[dr_0\wedge d\{r_0,r_1\}_{s-1}\wedge dr_2\wedge \cdots \wedge dr_{2n-1}(0)\ne 0.\]
\end{itemize} 
\end{definition}

\begin{remark}
\label{rem0Ss}
\normalfont{For $s=0$ we obtain diffeomorphisms, $S_0(\widetilde{\omega})=\text{Diff}(\R^{2n})$. For $n=1$ and $s=1$ we obtain ordinary folds $S_1$ (in the sense of Whitney c.f. \cite{A11}, \cite{GG}). Notice though that while ordinary cusps $S_2$ appear typically for mappings on the plane $\R^2\rightarrow \R^2$, they do not appear typically for Whitney-type mappings $(\R^2,\widetilde{\omega})\rightarrow \R^2$ as defined above (due to (a)), $S_2(\widetilde{\omega})=\emptyset$. Moreover, for any $n\geq 2$, the mappings $r\in S_{s}(\widetilde{\omega})$, $s\geq 1$, are not even Whitney mappings in the ordinary sense, unless one imposes in Definition \ref{defrSs} the following further condition:
\begin{itemize}
\item[(d)] \textit{the singular locus $\Sigma(r)=\{\det[dr]=0\}$ (the vanishing of the Jacobian determinant of $r$) is given by the hypersurface:}
\[\Sigma(r)=\{\{r_0,r_1\}=0\}.\] 
\end{itemize} 
Indeed, condition (a) then implies that the Hamiltonian vector field $Z_{r_0}:=\{r_0,\cdot\}$ of the function $r_0$ spans the $1$-dimensional kernel field of the differential of the map $r$ over its singular locus $\Sigma(r)$:
\[\ker dr=\text{span}\{r_0,\cdot\},\]
and has exactly $(s-1)$-order of tangency with it at the origin. By condition (b) we obtain then that the mapping $r\in S_{s}(\widetilde{\omega})$ is indeed a Whitney map in the ordinary sense, i.e. $r\in S_s$ (c.f. again \cite{A11}, \cite{GG}). Notice thought that the singularity class $S_{2n}$ is generic for ordinary Whitney mappings whereas, due to (a), it does not appear typically for Whitney-type mappings in symplectic space $S_{2n+1}(\widetilde{\omega})=\emptyset$.}
\end{remark}

\begin{remark}
\label{rem2Ss}
\normalfont{Notice also that conditions (a)-(c) define an open set in the space $s$-jets of all mappings $(\R^{2n},\widetilde{\omega})\rightarrow \R^{2n}$ satisfying conditions:
\[\{r_0,r_1\}_i(0)=0, \quad i=0,\cdots,s-1, \quad \{r_0,r_1\}_s(0)\ne 0,\] 
whereas condition (d) of being Whitney in the ordinary sense, is non-generic in the space of all mappings satisfying (a)-(c). }
\end{remark}

The $\mathcal{R}(\widetilde{\omega})$-classification of mappings $r\in S_{s}(\widetilde{\omega})$ with respect to a fixed symplectic structure $\widetilde{\omega}$ in the source, is naturally a sub-problem of the $\mathcal{R}$-classification of mappings $r\in S_s(\widetilde{\omega})$, i.e. under ordinary (non-symplectic) diffeomorphisms of the source space (recall that no changes of coordinates in the target are allowed), a problem which already contains functional moduli for all $s\geq 1$ (for $s=0$ it is trivial that any diffeomorphism $r$ is $\mathcal{R}$-equivalent to the identity). Indeed: 
\begin{proposition}
\label{proprSs}
Any mapping $r\in S_s(\widetilde{\omega})$, $1\leq s\leq 2n$, is $\mathcal{R}$-equivalent to the normal form:
\begin{equation}
\label{nfRSs}
r(p,q)=(p_1,q_1^{s+1}+\sum_{j=0}^{s-1}r_{1,j}(\widehat{q}_1)q_1^j,p_2+\widetilde{r}_{2}(p,q),\cdots,q_n+\widetilde{r}_{2n-1}(p,q)),
\end{equation}
where the functions $\widetilde{r}_i(p,q)$ belong in the ideal $\I_1:=<q_1>$, $\widetilde{r}_i\in \I_1$, $i=2,\cdots,2n-1$, and the functions $r_{1,j}(\widehat{q}_1)$, $r_{1,j}(0)=0$, $j=0,\cdots, s-1$, are differentially independent at the origin:
\begin{equation}
\label{dir1j}
dr_{1,0}\wedge \cdots \wedge dr_{1,s-1}(0)\ne 0.
\end{equation}
 The functions $\{r_i\in \mathcal{I}_1\}_{i=2}^{2n-2},\{r_{1,j}\}_{j=0}^{s-1}$, are functional moduli for mappings $r\in S_s(\widetilde{\omega})$ under $\mathcal{R}$-equivalence. 
\end{proposition}
\begin{proof}
 Since by condition (a) the function $r_0$ is non-singular, $dr_0(0)\ne 0$, we can rectify its corresponding Hamiltonian vector field $Z_{r_0}$ in some coordinates $(p,q)$ (not necessarily Darboux), such that $Z_{r_0}:=\{r_0,\cdot\}=\partial_{q_1}$. In these coordinates we obtain from condition (a) the following conditions:
\[\partial_{q_1}^ir_1(0)=0, \quad i=1,\cdots,s, \quad \partial_{q_1}^{s+1}r_1(0)\ne 0.\]
We may reduce now the function $r_1(p,q)$ by a diffeomorphism preserving the foliation $\text{span}\{\partial_{q_1}\}$ of the Hamiltonian vector field $Z_{r_0}=\partial_{q_1}$, to the classical preliminary normal form of a versal unfolding of an $A_s$-singularity $q_1^{s+1}$ (c.f. \cite{A00}):
\begin{equation}
\label{nfr1RSs}
r_1(p,q)=q_1^{s+1}+\sum_{j=0}^{s-1}r_{1,j}(\widehat{q}_1)q_1^j,
\end{equation}
for some functions $r_{1,j}(\widehat{q}_1)$, $r_{1,j}(0)\ne 0$, $j=0,\cdots,s-1$, which, from condition (b), are differentially independent at the origin, i.e. satisfy (\ref{dir1j}).
 From conditions (c) and (a) we obtain also:
\[dr_0\wedge d(\partial_{q_1}^sr_1)\wedge dr_2\wedge \cdots \wedge dr_{2n-1}(0),\]
from which we can suppose (up to renumeration) that the following conditions are satisfied:
\[\partial_{p_1}r_0(0)\ne 0 \hspace{0.2cm} \text{and} \hspace{0.2cm} \partial_{p_{i+1}}r_{2i}(0)\ne 0, \quad \partial_{q_{i+1}}r_{2i+1}(0)\ne 0, \quad i=1,\cdots,n-1.\]
Using now a diffeomorphism of the source preserving the coordinate $q_1$, as well as the foliation $\text{span}\{\partial_{q_1}\}$, i.e. of the form:
\[(p,q)\mapsto (P_1(\widehat{q}_1),q_1,P_2(\widehat{q}_1),\cdots,Q_n(\widehat{q}_1)),\]
we may reduce each of the functions $r_i(p,q)$, $i\ne 1$, to:
\[r_0(p,q)=p_1, \quad r_{2i}(p,q)=p_{i+1}+\widetilde{r}_{2i}(p,q), \quad r_{2i+1}(p,q)=q_{i+1}+\widetilde{r}_{2i+1}(p,q), \quad i=1,\cdots,n-1,\]
for some functions $\widetilde{r}_i(p,q)$ belonging in the ideal $\I_1:=<q_1>$,  $\widetilde{r}_i\in \I_1$, $i=2,\cdots,2n-1$. These, along with the (new) functions $r_{1,j}(\widehat{q}_1)$, $j=0,\cdots,s-1$, in normal form (\ref{nfr1RSs}) for $r_1(p,q)$, are obvious functional invariants by construction, which proves the proposition.
\end{proof}

\begin{remark}
\label{rem2Ss}
\normalfont{The functions $r_{1,j}(\widehat{q}_1)$, $j=0,\cdots,s-1$, in the normal form (\ref{nfr1RSs}) of $r\in S_{s}(\widetilde{\omega})$, are already invariants for the $\mathcal{R}$-classification  of Whitney mappings in the ordinary sense. Indeed, in this case we may reduce the mapping $r\in S_s(\widetilde{\omega})$ by a (non-symplectic) diffeomorphism in the source to:
\[r(p,q)=(p_1,r_1(p,q),p_2,\cdots,q_n),\]
and then from condition (d) we obtain:
\[\Sigma(r)=\{\{r_0,r_1\}=0\}=\{\partial_{q_1}r_1=0\}.\]
From conditions (a) and (b) we may reduce now the function $r_1(p,q)$ by some diffeomorphism $q_1\mapsto Q_1(p,q)$ fixing all other coordinates, to the standard normal form  (\ref{nfr1RSs}), and consequently the mapping $r\in S_s$ to the exact normal form:
\[r(p,q)=(p_1,q_1^{s+1}+\sum_{j=0}^{s-1}r_{1,j}(\widehat{q}_1)q_1^i,p_2,\cdots,q_n),\]
with the functional invariants $r_{1,j}(\widehat{q}_1)$, $j=0,\cdots,s-1$, being functionally independent at the origin (by (b)).}
\end{remark}

Fix now a symplectic structure $\widetilde{\omega}$ in the source $\R^{2n}$, given say by the Darboux normal form $\widetilde{\omega}=(\ref{dnft})$. Then the corresponding $\mathcal{R}(\widetilde{\omega})$-classification of Whitney-type mappings $r\in S_s(\widetilde{\omega})$, $s\geq 0$, has of course more functional moduli than the $\mathcal{R}$-classification. Below we show that these extra moduli can be naturally represented by functions in the whole nested sequence of invariant ideals $\I_i$, $i=1,\cdots,2n-1$, parametrising the space of all symplectic structures $\Omega^2_S(\R^{2n})$. The case of diffeomorphisms $s=0$ has been treated extensively in \cite{K1}. The same method of proof can be applied to the more general cases, $s\geq 1$ as well:  
\begin{theorem}
\label{thm-wk}
Any $S_s(\widetilde{\omega})$-singularity $r:(\R^{2n},\widetilde{\omega})\rightarrow \R^{2n}$, $s\geq 0$, is equivalent to the normal form:
\[r(p,q)=(p_1,r_1(p,q),\cdots,r_{2n-1}(p,q)),\]
where:
\begin{equation}
\label{nfroddWk}
r_{2m+1}\in \I_{2m+1}, \quad \partial_{q_{m+1}}r_{2m+1}(0)\ne 0, \quad m=1,\cdots,n-1,
\end{equation}
\begin{equation}
\label{nfrevenWk}
r_{2m}(p,q)=p_{m+1}+\widetilde{r}_{2m}(p,q), \quad \widetilde{r}_{2m}\in \I_{2m}, \quad m=1,\cdots,n-1,
\end{equation}
and:
\begin{equation}
\label{nfr1Wk}
r_1(p,q)=\psi(p,q)q_1^{s+1}+\sum_{j=0}^{s-1}r_{1,j}(\widehat{q}_1)q_1^{j},
\end{equation}
where $\psi(0)\ne 0$, $r_{1,j}(0)=0$, $j=0,\cdots,s-1$, and the $s$-functions $r_{1,j}$ are differentially independent at the origin:
\begin{equation}
\label{dir1jb}
dr_{1,0}\wedge \cdots \wedge dr_{1,s-1}(0)\ne 0.
\end{equation}
The set of functions:
\begin{equation}
\label{msr}
\text{m}_{s}(r)=\big \{\psi,\{r_{1,j}\}_{j=0}^{s-1},\{r_{2m+1}\in \I_{2m+1}\}_{m=1}^{n-1},\{\widetilde{r}_{2m}\in \I_{2m}\}_{m=1}^{n-1} \big \}
\end{equation}
is a complete set of functional moduli for the classification of Whitney-type mappings in $S_s(\widetilde{\omega})$: two mappings $r$, $r'$ in $S_s(\widetilde{\omega})$ are $\mathcal{R}(\widetilde{\omega})$-equivalent if and only if $\text{m}_{s}(r)\equiv \text{m}_{s}(r')$.
\end{theorem}

\begin{proof}
We denote by $r(p,q)=(r_0(p,q),\cdots,r_{2n-1}(p,q))$ the coordinate functions of our initial map $r\in S_s(\widetilde{\omega})$.  Since $dr_0(0)\ne 0$, we may write $r_0(p,q)=p_1$ by a symplectomorphism of $\widetilde{\omega}=dp\wedge dq$, and then the corresponding Hamiltonian vector field of $r_0$ is also rectified $Z_{p_1}=\partial_{q_1}$. Denote now by: 
\[\Sigma_{s}(r_1):=\{\{r_0,r_1\}_{s-1}=0\}=\{\partial^s_{q_1}r_1=0\}, \hspace{0.2cm} \text{for} \hspace{0.2cm} s\geq 1,\]
\[(\text{or}\hspace{0.2cm} \Sigma_{0}(r_1)=\{r_1=0\}, \hspace{0.2cm} \text{for} \hspace{0.2cm} s=0,)\]
the ``$s$-singular locus" of the function $r_1$. It is a smooth hypersurface which, by Theorem \ref{thm1} on non-singular sections, can be put into normal form:
\[\Sigma_s(r_1)=\{q_1=0\},\]
by a symplectomorphism of the pair $(\widetilde{\omega},p_1)$.  We can continue now the normalisation of the rest of the functions $(r_2,\cdots,r_{2n-1})$ by symplectomorphisms preserving the triple $(\widetilde{\omega},p_1,\Sigma_s(r_1)=\{q_1=0\})$ (or what is equivalent, the triple $(\widetilde{\omega},p_1,\I_1=<q_1>)$). These, as is easily verified, are necessarily mappings of the form:
\[(p,q)\mapsto (p_1,q_1,\Phi(\widehat{q}_1,\widehat{p}_1)), \quad \Phi^*(\omega|_{p_1=q_1=0})=\omega|_{p_1=q_1=0},\] 
where $\Phi(\widehat{q}_1,\widehat{p}_1)$ is a symplectomorphism of the restriction $\widetilde{\omega}|_{p_1=q_1=0}=\sum_{i=2}^{n}dp_i\wedge dq_i$ of $\widetilde{\omega}$ on the symplectic subspace $p_1=q_1=0$. With such a symplectomorphism we can reduce the second pair of functions $(r_2,r_3)$ to the preliminary normal form:
\[r_2(p,q)=p_2+\widetilde{r}_2(p,q), \hspace{0.2cm} \text{where} \hspace{0.2cm} \widetilde{r}_2\in \I_2=<q_1,p_1>,\] 
\[r_3\in \I_3=<q_1,p_1,q_2>, \quad \partial_{q_2}r_3(0)\ne 0,\]
by taking restriction on the $q_1=p_1=0$ symplectic subspace and using again Theorem \ref{thm1} on non-singular sections. Continuing in the same way, i.e. by taking successive restrictions of the rest of the pairs $(r_{2i},r_{2i+1})$, $i=2,\cdots,n-1$, on the symplectic subspaces $p_1=q_1=\cdots=p_i=q_i=0$, and normalising, using Theorem \ref{thm1}, with respect to the induced symplectic forms, we obtain the required result. In particular, by the normal form $\Sigma_s(r_1)=\{q_1=0\}$ of the $s$-singular locus, we deduce that the function $r_1$ is also reduced to the exact normal form: 
\[r_1(p,q)=\psi(p,q)q_1^{s+1}+\sum_{i=0}^{s-1}r_{1j}(\widehat{q}_1)q_1^i,\]
with the functional invariants $\psi(p,q)$, $\psi(0)\ne 0$, and $r_{1j}(\widehat{q}_1)$, $r_{1j}(0)=0$, $j=0,\cdots,s-1$. The latter satisfy also (\ref{dir1jb}) by condition (b) in Definition \ref{defrSs}.  
\end{proof}

From this, and Remark \ref{rem2Ss}, we immediately obtain the following description of the corresponding moduli space $\mathcal{M}(S_s(\widetilde{\omega}))$ of Whitney-type mappings under $\mathcal{R}(\widetilde{\omega})$-equivalence:

\begin{corollary}
\label{corwk}
The exists an isomorphism:
\[\mathcal{M}(S_s(\widetilde{\omega}))\cong \Omega^2_S(\R^{2n})\times \mathcal{M}(S_s)\cong C^{\infty}(\R^{2n})^{n(2n-1)}\times C^{\infty}(\R^{2n-1})^s,\]
where $\mathcal{M}(S_{s})$ is the moduli space of Whitney mappings $S_s$ under ordinary $\mathcal{R}$-equivalence. 
\end{corollary}

\section{Proofs of Theorems \ref{thm2}-\ref{thm7}}
\label{sec4}

\subsection{Proof of Theorem \ref{thm2}}
\label{sec4:1}
Since the Hamiltonian vector field $Z_f=\partial_x$ of $(\omega,f)=(\ref{nfHS})$ has $k$-order tangency with the hypersurface $H$, we can write by the Malgrange-Weierstrass preparation theorem:
\[H=\{x^{k+1}+\sum_{i=0}^kR_i(y,p,q)x^i=0\},\]
for some functions $R_i(y,p,q)$ vanishing at the origin, $R_i(0)=0$, $i=0,\cdots,k$. Consider now the $k$-singular locus of the pair $(H,Z_f)$:
\[C_{k}:=\{\{f,h\}_{k-1}:=Z_f^{k}(h)=0\}=\{(k+1)x+R_{k}(y,p,q)=0\}.\]
The change of coordinates $x\mapsto x-(1/(k+1))R_{k}(y,p,q)$, preserves $Z_f=\partial_x$, brings $C_{k}$ to the normal form:
\[C_{k}=\{x=0\},\]
and the hypersurface $H$ to:
\[H=\{x^{k+1}+\sum_{i=0}^{k-1}R_i(y,p,q)x^i=0\},\]
for some (new) functions $R_i(y,p,q)$, $R_i(0)=0$, $i=0,\cdots,k-1$. In these coordinates the symplectic form $\omega$ is also reduced to:
\[\omega=dx \wedge dy+\widehat{\omega},\]
where $\widehat{\omega}$ is a closed $2$-form such that $\partial_x\lrcorner \widehat{\omega}=0$, i.e. it is defined on the orbit space $\R^{2n+1}_{(y,p,q)}$ of the Hamiltonian vector field $Z_f=\partial_x$. By the non-degeneracy of $\omega$, we obtain from $\omega^{n+1}(0)=dx\wedge dy\wedge \widehat{\omega}^n(0)\ne 0$, that $\widehat{\omega}$ is a quasi-symplectic structure (a closed $2$-form of maximal rank $2n$) in $\R^{2n+1}_{(y,p,q)}$, which moreover satisfies $dy\wedge \widehat{\omega}^n(0)\ne 0$, i.e. its $1$-dimensional kernel field $\ker \widehat{\omega}$ is transversal to the fibers $X_t=\{y=t\}$ of $f=y$ at the origin (viewed as a function in the orbit space $\R^{2n+1}_{(y,p,q)}$ of its Hamiltonian vector field  $Z_f=\partial_x$). Using now a diffeomorphism of the form:
\[(y,p,q)\mapsto (y,\Phi(y,p,q)),\] 
it is easy to show (by an odd-dimensional version of the Darboux theorem, c.f. \cite{Z}) that we may reduce $\widehat{\omega}$ back to its standard Darboux normal form $\widehat{\omega}=dp\wedge dq$, which leads to the required normal form:
\[\omega=dx\wedge dy+dp\wedge dq, \quad f=y,\]
\[H=\{x^{k+1}+\sum_{i=0}^{k-1}R_{i}(y,p,q)x^{i}=0\},\]
for some (new) functions $R_{i}(y,p,q)$,  $R_i(0)=0$, $i=0,\cdots,k-1$. Finally, conditions (\ref{ctr'})-(\ref{cSklR'}) follow after simple calculations by the definitions as is explained in Remark \ref{rem0}, which finishes the proof. 
\qed

\subsection{Proof of Theorem \ref{thm3}}
\label{sec4:2}

This case reduces to find an exact normal form for the associated coefficient mappings $R(y,p,q)=(R_0(y,p,q),\cdots,R_{k-1}(y,p,q))$ of generic sections $H\in S_{k,1}$, $k\geq 1$, by (quasi-)symplectomorphisms of the pair $(y,dp\wedge dq)$, under transversality condition (\ref{ctr'}), and also conditions (\ref{cS11R'})-(\ref{cSklR'}) for $l=1$, appearing in Theorem \ref{thm2}. Expand now each of the functions $R_i(y,p,q)$, $i=0,\cdots,k-1$, as a Taylor series in the variable $y$, in the following way:
\[R_i(y,p,q)=r_i(p,q)+\phi_i(y,p,q)y, \quad i=1,\cdots,k-1,\quad \text{and}\]
\[R_0(y,p,q)=g_0(y)+r_0(p,q)+\sum_{j\geq 1}^{2n-k}r_{k-1+j}(p,q)y^j+\phi_0(y,p,q)y^{2n-k+1}.\]
If we denote by $r(p,q)=(r_0(p,q),\cdots,r_{2n-1}(p,q))$ the associated mapping of first $2n$-coefficients in the above expansions, we see that conditions (\ref{ctr'})-(\ref{cSklR'}) for $l=1$, imply that for any $k\geq 1$, the first pair of functions $(r_0,r_1)$ is non-singular, i.e. it satisfies 
\[\{r_0,r_1\}(0)\ne 0.\] 
With this condition fixed, we can suppose that the rest of the coefficient functions $r_i$, $i=2,\cdots,2n-1$, can be chosen so that the associated map $r=(r_0,\cdots,r_{2n-1})$ is a diffeomorphism $r\in S_{0}(\widetilde{\omega})$ in the symplectic space $(\R^{2n}_{(p,q)},dp\wedge dq)$, a condition that defines the open set $U$ in the statement of the theorem. The proof is concluded by reducing the mappings $r\in S_0(\widetilde{\omega})$ to the exact normal form of Theorem \ref{thm-wk}, for $s=0$.

\qed
  
\subsection{Proof of Theorem \ref{thm4}}
\label{sec4:3}

This case reduces to the previous one under the generic assumption that the whole coefficient mapping $R(y,p,q)=(R_0(y,p,q),\cdots,R_{2n}(y,p,q))$ also defines a diffeomorphism in quasi-symplectic space, i.e. 
\[dR_0\wedge \cdots, \wedge dR_{2n}(0)\ne 0.\]  
Dividing the first $2n$-functions $R_i(y,p,q)$, $i=0,\cdots,2n-1$, with the ideal $<y>$ we obtain:
\[R_i(y,p,q)=r_i(p,q)+\phi_i(y,p,q)y, \quad i=0,\cdots, 2n-1,\]
and as before, the associated mapping $r(p,q)=(r_0(p,q),\cdots,r_{2n}(p,q))$ defines a diffeomorphism $r\in S_0(\widetilde{\omega})$ in the symplectic space $(\R^{2n}_{(p,q)},dp\wedge dq)$. The result now follows  again by applying Theorem \ref{thm-wk} to the mapping $r\in S_0(\widetilde{\omega})$.

\qed

\subsection{Proof of Theorem \ref{thm5}}
\label{sec4:4}

This case reduces to find exact normal form for the single function $R_0(y,p,q)$, by (quasi-)symplectomorphisms of the pair $(y,dp\wedge dq)$,  under conditions (\ref{ctr'}) and (\ref{cS1lR'}) appearing in Theorem \ref{thm2}.
Proceeding as in the proof of Theorems \ref{thm3}-\ref{thm4} above, we expand the function $R_0(y,p,q)$, $i=0,\cdots,k-1$, as a Taylor series in $y$:
\[R_0(y,p,q)=g_0(y)+r_0(p,q)+\sum_{i\geq 1}^{2n-1}r_{i}(p,q)y^i+\phi_0(y,p,q)y^{2n}.\]
Conditions (\ref{ctr'}) and (\ref{cS1lR'}), imply that the first pair of functions $(r_0,r_1)$ in the expansions above satisfies:
\[\{r_0,r_1\}_i(0)=0, \forall i=0,\cdots,l-3, \quad \{r_0,r_1\}_{l-2}(0)\ne 0.\]
With these conditions fixed, we can suppose that the rest of the coefficient functions $r_i$, $i=2,\cdots,2n-1$, are chosen so that the associated map $r(p,q)=(r_0(p,q),\cdots, r_{2n-1}(p,q))$ defines a Whitney-type map $r\in S_{l-2}(\widetilde{\omega})$, $l\geq 2$, in the symplectic space $(\R^{2n}_{(p,q)},dp\wedge dq)$ in the sense of Definition \ref{defrSs} in Section \ref{sec3}, a condition that defines the open set $U$ in the statement of the theorem. The proof follows now by applying Theorem \ref{thm-wk} to  the map $r\in S_{l-2}(\widetilde{\omega})$.

\subsection{Proof of Theorem \ref{thm6}}
\label{sec4:5}

Working as in the proofs of Theorems \ref{thm3}-\ref{thm5} above, we expand each of the functions $R_i(y,p,q)$, $i=0,\cdots,k-1$, as a Taylor series in $y$ in the following way:
\[R_i(y,p,q)=r_i(p,q)+\phi_i(y,p,q)y, \quad i=1,\cdots,k-1,\]
\[R_0(y,p,q)=g_0(y)+r_0(p,q)+\sum_{j\geq 1}^{2n-k}r_{k-1+j}(p,q)y^j+\phi_0(y,p,q)y^{2n-k+1}.\]
Conditions (\ref{ctr'}) and (\ref{cSklR'}) appearing in Theorem \ref{thm2}, imply that for $l\geq 2$ the first pair of functions $(r_0,r_1)$ in the expansions above satisfies:
\[\{r_0,r_1\}_i(0)=0, \forall i=0,\cdots,l-2, \quad \{r_0,r_1\}_{l-1}(0)\ne 0.\]
With these conditions fixed, we can suppose that the rest of the coefficient functions $r_i$, $i=2,\cdots,2n-1$, are chosen so that the associated map $r(p,q)=(r_0(p,q),\cdots, r_{2n-1}(p,q))$ defines a Whitney-type map $r\in S_{l-1}(\widetilde{\omega})$, $l\geq 2$, in the symplectic space $(\R^{2n}_{(p,q)},dp\wedge dq)$ in the sense of Definition \ref{defrSs}, a condition that defines the open set $U$ in the statement of the theorem. The proof follows again by applying Theorem \ref{thm-wk} to  the map $r\in S_{l-1}(\widetilde{\omega})$.

\qed

\subsection{Proof of Theorem \ref{thm7}}
\label{sec4:6}

Since $\{f,h\}(0)=0$ and $\{f,\{f,h\}(0)\ne 0$, we may, following exactly the same proof of Theorem \ref{thm2}, reduce $H\in A_1$ to the preliminary normal form:
\[H=\{x^2+R(y,p,q)=0,\}\] 
by a symplectomorphism of $(\omega,f)=(\ref{nfHS})$. Since $\{h,\{h,f\}(0)\ne 0$ we obtain $\partial_yR(0)\ne 0$. Notice now that the condition $df\wedge dh(0)=0\Longleftrightarrow dy\wedge dR(0)=0$ is equivalent to the condition:
\[dR|_{y=0}(0)=0,\]
i.e the restriction of the function $R(y,p,q)$ on the 0-fiber $X_0=\{y=0\}$ of $f=y$, has a critical point at the origin. For a generic section $H\in A_1$, this critical point will be of Morse type, i.e. 
\[d^2R|_{y=0}(0)\ne 0,\]
a condition equivalent to condition (\ref{cA1}) in Definition \ref{def4}. Thus, the problem to obtain an exact normal form for sections $H\in A_1$, reduces to find an exact normal form for the function $R(y,p,q)$ under symplectomorphisms of the pair $(y,dp\wedge dq)$. Dividing by the ideal $<y>$ we obtain a decomposition:
\[R(y,p,q)=\psi(p,q)+\phi(y,p,q)y,\]
where $\phi(0)\ne 0$ and $\psi(p,q)$ is a Morse function in the symplectic space $(\R^{2n},dp\wedge dq)$. The theorem follows now by bringing the function $\psi$ to an exact normal form $\widetilde{\psi}$ with respect to $dp\wedge dq$.

\qed

\section*{Acknowledgements}

This research has been supported by the S\~ao Paulo Research Foundation, FAPESP, grant no.: 2017/23555-9.

\end{document}